\documentclass[12pt,twoside]{article}
\usepackage{graphicx}
\usepackage{amsfonts, amsthm, amsmath, amssymb}
\usepackage{hyperref}
\usepackage{longtable}
\allowdisplaybreaks
\usepackage[nottoc]{tocbibind} 

\date{}

\begin{document}

\centerline{}

\centerline{}

\centerline {\Large{\bf On the Interval $[n,2n]$:}}

\centerline{}

\centerline{\Large{\bf Primes, Composites and Perfect Powers}}

\centerline{}

\newcommand{\mvec}[1]{\mbox{\bfseries\itshape #1}}

\centerline{\bf {Germ\'an Andr\'es Paz}}

\centerline{}

\centerline{Instituto de Educaci\'on Superior N$^\circ$28 Olga Cossettini,}
\centerline{(2000) Rosario, Santa Fe, Argentina}
\centerline{E-mail: germanpaz\_ar@hotmail.com}

\centerline{}

\newtheorem{Theorem}{\quad Theorem}[section]

\newtheorem{Definition}[Theorem]{\quad Definition}

\newtheorem{Corollary}[Theorem]{\quad Corollary}

\newtheorem{Lemma}[Theorem]{\quad Lemma}

\newtheorem{Example}[Theorem]{\quad Example}

\newtheorem{Remark}[Theorem]{\quad Remark}

\newtheorem{Proposition}[Theorem]{\quad Proposition}
\centerline{}

\centerline{\bf Abstract}
{\emph{In this paper we show that for every positive integer $n$ there exists a prime number in the interval $[n,9(n+3)/8]$. Based on this result, we prove that if $a$ is an integer greater than
1, then for every integer $n > 14.4a$ there are at least four
prime numbers $p$, $q$, $r$, and $s$ such that $n < ap
< 3n/2 < aq < 2n$ and $n < r < 3n/2 <
s < 2n$. Moreover, we also prove that if $m$ is a positive integer, then for
every positive integer $n > 14.4/{{{{\left( {\left|\sqrt[m]{{1.5}}\right| - 1}
\right)}^m}}}$ there exist a positive integer $a$ and a prime number
$s$ such that $n < {a^m} < 3n/2 < s < 2n$, as well as the fact that for every positive integer $n > 14.4/{{{{\left( {\left|\sqrt[m]{2}\right| -
\left|\sqrt[m]{{1.5}}\right|} \right)}^m}}}$ there exist a prime number $r$ and a
positive integer $a$ such that $n < r < 3n/2 < {a^m} < 2n$.}}


\section{Introduction}
\label{intro}

The well-known Bertrand's postulate, which was first proved by P. L. Chebyshev in 1850, states that for every integer $n>3$ there exists a prime number $p$ such that $n<p<2n-2$. Another formulation of this theorem, though weaker, is that for every integer $n>1$ there exists a prime number $p$ such that $n<p<2n$. Nowadays, these results are known as \emph{Bertrand-Chebyshev theorem} or \emph{Chebyshev's theorem}.

Generalizations of Bertrand-Chebyshev theorem and better results have been obtained by using both elementary and nonelementary methods. In 1998, Pierre Dusart \cite{Dusart} proved that for every $n \ge 3275$ there exists a prime in the interval $[n,(1+1/(2 \ln^2 n))n]$. He improved this result in 2010 \cite{Dusart2}, when he proved that for every $n \ge 396738$ there exists a prime in the interval $[n,(1+1/(25 \ln^2 n))n]$. In 2006, M. El Bachraoui \cite{M. El Bachraoui} had proved that there exists a prime in the interval $[2n,3n]$, and in 2011 Andy Loo \cite{Loo} proved that there exists a prime in the interval $[3n,4n]$. In Sect. \ref{sec:7} we will see that Theorem \ref{Theorem3} implies these results obtained by M. El Bachraoui and Andy Loo.

In 2009, Carlos Giraldo Ospina (Lic. Matem\'aticas, USC, Cali, Colombia) proposed a proof of Polignac's conjecture in his web article entitled \emph{Primos Gemelos, Demostraci\'on Kmelliza}. In some of his web articles, C. Giraldo Ospina uses a so-called \emph{Breusch's interval} to propose proofs of general mathematical conjectures. This author assumes that Breusch's interval $[n,9(n+3)/8]$ contains at least one prime number for every positive integer $n$.

The purpose of this work is to prove that there exists a prime in the interval $[n,9(n+3)/8]$ for every positive integer $n$ and to use this result to study the interval $[n,2n]$, not only as regards prime numbers, but also as regards composite numbers and perfect powers in the mentioned interval. We will start by giving a proof that $\pi [n,9(n+3)/8] \ge 1\text{, }\forall{n}\in{\mathbb{Z}^+}$ (see Sects. \ref{sec:1}--\ref{sec:6}). The interval $[n,9(n+3)/8]$, which is easier to manipulate than other intervals, is later used in this paper to give proofs of the main results, which have already been stated in the abstract of this work. Please see Sects. \ref{sec:7}, \ref{sec:10}, and \ref{sec:11} for proofs of Theorems \ref{Theorem4}, \ref{Theorem9}, and \ref{Theorem12} respectively.

\section{Dusart's Interval}
\label{sec:1}
%
\begin{Definition}
\label{definition1}
In 1998, Pierre Dusart \emph{\cite{Dusart}} proved that for every $n \ge 3275$ there is at least one prime number in the interval
\begin{equation*}
\label{equation1}
\left[ {n,\left( {1+\frac{1}{2 \ln^2 n}} \right)n} \right]\text{,}
\end{equation*}
where $\ln$ denotes the \emph{natural logarithm}. The mentioned interval will be referred to as \emph{Dusart's interval}.
\end{Definition}
Since $\ln^2 n\equiv(\ln n)^2$, it follows that Dusart's interval can also be expressed as
\begin{align*}
\left[ {n,\left( {1+\frac{1}{2 (\ln n)^2}} \right)n} \right]\text{.}
\end{align*}
%
\section{Breusch's Interval}

\begin{Definition}
The interval $[n,9(n+3)/8]$, where $n$ is a positive integer, is what we will call \emph{Breusch's interval}.
\end{Definition}
In Sect. \ref{section4} we are going to prove that there is at least one prime number in Breusch's interval for every positive integer $n$. This statement may have been proved a long time ago; however, we will need to provide a clear proof that it is true.

\section{Breusch's Interval Containing at least One Dusart's Interval}
\label{section4}

\begin{Lemma}
\label{lemma111}
It is easy to verify that the value of $2(\ln n)^2$ increases as the value of $n$ increases, which means that the value of $1/(2(\ln n)^2)$ decreases as $n$ increases. All this implies that the value of $1+1/(2(\ln n)^2)$ decreases as the value of $n$ increases.

If we consider $n=e$, where $e$ is the base of the natural logarithm, we have
\begin{align*}
1+\frac{1}{2(\ln n)^2}=1+\frac{1}{2(\ln e)^2}=1+\frac{1}{2}=1.5\text{.}
\end{align*}
Now, if we consider $n=e^2$, we have
\begin{align*}
1+\frac{1}{2(\ln n)^2}=1+\frac{1}{2(\ln e^2)^2}=1+\frac{1}{2 \cdot 2^2}=1+\frac{1}{8}=1.125\text{.}
\end{align*}
\end{Lemma}

\begin{Lemma}
\label{lemma222}
Breusch's interval $[n,9(n+3)/8]$ can also be expressed as $[n,1.125n+3.375]$, and of course for every positive integer $n$ we have
\begin{align*}
1.125n<1.125n+3.375\text{.}
\end{align*}
\end{Lemma}
Combining Lemmas \ref{lemma111} and \ref{lemma222}, we obtain the following result, which we will express as a new lemma:

\begin{Lemma}
\label{lemma3}
For every $n \ge e^2$ we have
\begin{align*}
n<\left( {1+\frac{1}{2(\ln n)^2}} \right)n<1.125n+3.375\text{,}
\end{align*}
which means that for every integer $n>e^2$, a Breusch's interval contains at least one Dusart's interval.
\end{Lemma}

As we already know, Dusart's interval contains at least one prime number for every $n \ge 3275$ (see Definition \ref{definition1}). If we combine this true statement with Lemma \ref{lemma3}, we conclude that for every integer $n \ge 3275$ a Breusch's interval contains at least one Dusart's interval containing at least one prime number. This means that Breusch's interval $[n,9(n+3)/8]$ contains at least one prime number for every integer $n \ge 3275$. With the help of computer software it can also be proved that there is a prime number in Breusch's interval for every integer $n$ such that $1 \le n \le 3274$:


\begin{longtable}{|c|c|c|}
\hline
   {$n$} & Example of prime in the interval $[n,9(n+3)/8]$ &          $9(n+3)/8$ \\
\hline \hline
\endhead
\multicolumn{2}{c}{Continued on next page...}
\endfoot
\endlastfoot
\hline
   {\bf 1} &          2 &  {\bf 4.5} \\
\hline
   {\bf 2} &          2 & {\bf 5.625} \\
\hline
   {\bf 3} &          3 & {\bf 6.75} \\
\hline
   {\bf 4} &          5 & {\bf 7.875} \\
\hline
   {\bf 5} &          5 &    {\bf 9} \\
\hline
   {\bf 6} &          7 & {\bf 10.125} \\
\hline
   {\bf 7} &          7 & {\bf 11.25} \\
\hline
   {\bf 8} &         11 & {\bf 12.375} \\
\hline
   {\bf 9} &         11 & {\bf 13.5} \\
\hline
  {\bf 10} &         11 & {\bf 14.625} \\
\hline
  {\bf ...} &         ... & {\bf ...} \\
\hline
{\bf 3265} &       3271 & {\bf 3676.5} \\
\hline
{\bf 3266} &       3271 & {\bf 3677.625} \\
\hline
{\bf 3267} &       3271 & {\bf 3678.75} \\
\hline
{\bf 3268} &       3271 & {\bf 3679.875} \\
\hline
{\bf 3269} &       3271 & {\bf 3681} \\
\hline
{\bf 3270} &       3271 & {\bf 3682.125} \\
\hline
{\bf 3271} &       3271 & {\bf 3683.25} \\
\hline
{\bf 3272} &       3299 & {\bf 3684.375} \\
\hline
{\bf 3273} &       3299 & {\bf 3685.5} \\
\hline
{\bf 3274} &       3299 & {\bf 3686.625} \\
\hline
\end{longtable}


\section{Breusch's Interval Conclusion}
\label{sec:6}

We have proved that Breusch's interval $[n,9(n+3)/8]$ contains at least one prime number for every positive integer $n$.
\begin{align*}
\boxed{
\pi \left[ {n,\frac{{9\left( {n + 3} \right)}}{8}} \right] \ge 1\text{, }\forall{n}\in{\mathbb{Z}^+}}
\end{align*}

\section{Theorems \ref{Theorem2}, \ref{Theorem3}, and \ref{Theorem4}}
\label{sec:7}

In Sect. \ref{section4} we have proved that for every positive integer $n$ there is always a
prime number in Breusch's interval $[n,9(n+3)/8]$. In this section, we are going to use this result to find out two things:
\begin{itemize}
	\item the value that a positive integer $n$ can have so that there always exists a prime number $p$ such that $n/a < p < 3n/2a$, and
	\item the value that a positive integer $n$ can have so that there always exists a prime number $q$ such that $3n/2a<q<2n/a$.
\end{itemize}

\begin{Remark}
In this paper, whenever we say that a number $b$ is
between a number $a$ and a number $c$, it means that
$a < b < c$, which means that $b$ is never equal to $a$ or
$c$ (the same rule is applied to intervals). Moreover, the number
$n$ that we use in this paper is always a positive integer.
\end{Remark}

\subsection{Breusch's interval between $n/a$ and $3n/2a$}
\label{subsec:1}

Let us suppose that $a$ is a positive integer. We need to know the value that $n$ can have so that there exists a prime number $p$ such that $n/a<p<3n/2a$. If
\begin{equation*}
\frac{n}{a} < p < \frac{{3n}}{{2a}}\text{,}
\end{equation*}
then
\begin{equation*}
n<ap<\frac{3n}{2}\text{.}
\end{equation*}
In order to achieve our goal, we need to determine the value that $n$ can have so that there exists at least one Breusch's interval (and thus a prime
number) between $n/a$ and $3n/2a$.

The integer immediately following the number $n/a$ will be denoted by
$n/a+d$. In other words, $n/a+d$ is the smallest integer
that is greater than $n/a$:
\begin{itemize}
	\item If $n/a$ is an integer, then $n/a+d=n/a+1$,
because in this case we have $d=1$.
	\item If $n/a$ is not an integer, then in the expression $n/a+d$ we
have $0<d<1$. We do not have any way of knowing the exact value of
$d$ if we do not know the value of $n/a$ first.
\end{itemize}
The reason why we work in this way is because the number $n$ that we use in Breusch's interval is always a positive integer. Now, let us make the calculation. We have
\begin{align*}
\frac{{9\left( {\frac{n}{a} + d + 3} \right)}}{8} &< \frac{{3n}}{{2a}}\\
9\left( {\frac{n}{a} + d + 3} \right) &< 8 \times \frac{{3n}}{{2a}}\\
9\left( {\frac{n}{a} + d + 3} \right) &< \frac{{24n}}{{2a}}\\
9\left( {\frac{n}{a} + d + 3} \right) &< \frac{{12n}}{a}\\
\frac{{9n}}{a} + 9d + 27 &< \frac{{12n}}{a}\\
9d + 27 &< \frac{{12n}}{a} - \frac{{9n}}{a}\\
9d + 27 &< \frac{{3n}}{a}\text{.}\\
\intertext{We need to pick the largest possible value of $d$, which is $d=1$ (if $9 \times 1 + 27 < 3n/a$, then $9d + 27 < 3n/a$ for all
$d$ such that $0 < d \le 1$):}
9 \times 1 + 27 &< \frac{{3n}}{a}\\
9 + 27 &< \frac{{3n}}{a}\\
36 &< \frac{{3n}}{a}\\
\frac{{3n}}{a} &> 36\\
\frac{n}{a} &> \frac{{36}}{3}\\
\frac{n}{a} &> 12\\
n&>12a\text{.}
\end{align*}
So, if $n>12a$, then there is at least one prime number $p$ such that
$n/a<p<3n/2a$. This means that if $n>12a$, then there
exists a prime number $p$ such that
\begin{equation*}
n < ap < \frac{{3n}}{2}.
\end{equation*}
We will state this result as a lemma:

\begin{Lemma}
\label{lemma1}
If $a$ is a positive integer, then for every integer $n>12a$ there exists at least one prime number $p$ such that $n<ap<3n/2$.
\end{Lemma}

\subsection{Breusch's interval between $3n/2a$ and $2n/a$}
\label{subsec:2}

Now we need to calculate the value that a positive integer $n$ can have so that there is
at least one Breusch's interval (and thus a prime number) between
$3n/2a$ and $2n/a$. The integer immediately following the
number $3n/2a$ will be denoted by $3n/2a+d$. In other words, $3n/2a+d$ is the smallest integer that is
greater than $3n/2a$. Let us take into account that $d = 1$ or $0 < d
< 1$ depending on whether $3n/2a$ is an integer or not respectively. Let us make the calculation. We have
\begin{align*}
\frac{{9\left( {\frac{{3n}}{{2a}} + d + 3} \right)}}{8} &< \frac{{2n}}{a}\\
9\left( {\frac{{3n}}{{2a}} + d + 3} \right) &< 8 \times \frac{{2n}}{a}\\
9\left( {\frac{{3n}}{{2a}} + d + 3} \right) &< \frac{{16n}}{a}\\
\frac{{27n}}{{2a}} + 9d + 27 &< \frac{{16n}}{a}\\
9d + 27 &< \frac{{16n}}{a} - \frac{{27n}}{{2a}}\\
9d + 27 &< \frac{{2 \times 16n - 27n}}{{2a}}\\
9d + 27 &< \frac{{32n - 27n}}{{2a}}\\
9d + 27 &< \frac{{5n}}{{2a}}\text{.}\\
\intertext{We need to pick the largest possible value of $d$, which is $d=1$
(if $9 \times 1 + 27 < 5n/2a$, then $9d + 27 < 5n/2a$ for
all $d$ such that $0 < d \le 1$):}
9 \times 1 + 27 &< \frac{{5n}}{{2a}}\\
9 + 27 &< \frac{{5n}}{{2a}}\\
36 &< \frac{{5n}}{{2a}}\\
\frac{{5n}}{{2a}} &> 36\\
5n &> 36 \times 2a\\
5n &> 72a\\
n &> \frac{{72a}}{5}\\
n &> 14.4a\text{.}
\end{align*}
So, if $n > 14.4a$, then there is at least one prime number $q$ such that
$3n/2a<q<2n/a$. This means that if $n>14.4a$, then
there exists a prime number $q$ such that
\begin{equation*}
\frac{{3n}}{2} < aq < 2n\text{.}
\end{equation*}
We will state this result as a lemma:

\begin{Lemma}
\label{lemma2}
If $a$ is a positive integer, then for every integer $n>14.4a$ there exists at least one prime number $q$ such that $3n/2<aq<2n$.
\end{Lemma}
In Lemma \ref{lemma1} we have $n>12a$, while in Lemma \ref{lemma2} we have $n>14.4a$. Now, if $n>14.4a$, then $n>12a$. In other words,
\begin{equation}
\label{eq71}
n>14.4a \Rightarrow n>12a\text{.}
\end{equation}

Lemmas \ref{lemma1} and \ref{lemma2}, and implication (\ref{eq71}) lead to the following theorem:
\begin{Theorem}
\label{Theorem2}
If $a$ is any positive integer, then for every positive
integer $n > 14.4a$ there is always a pair of prime numbers
$p$ and $q$ such that $n < ap < 3n/2 < aq < 2n$.
\end{Theorem}

From Theorem \ref{Theorem2} we deduce the following theorem:
\begin{Theorem}
\label{Theorem3}
For every integer $n > 14.4$ there always exist prime numbers $r$ and $s$ such that $n < r < 3n/2 < s <
2n$.
\end{Theorem}
If $n > 14.4a$, then there is always a pair of prime numbers $p$ and
$q$ such that
\begin{equation*}
n < ap < \frac{{3n}}{2} < aq < 2n\text{,}
\end{equation*}
according to Theorem \ref{Theorem2}. If $n > 14.4$, then there is always a pair of prime numbers $r$ and
$s$ such that
\begin{equation*}
n < r < \frac{{3n}}{2} < s < 2n\text{,}
\end{equation*}
according to Theorem \ref{Theorem3}.

Now, if $n > 14.4a$ and $a > 1$, then $n > 14.4$. As a consequence, we state
the following theorem:

\begin{Theorem}
\label{Theorem4}
If $a$ is a positive integer greater than
1, then for every positive integer $n > 14.4a$ there are always at least four
prime numbers $p$, $q$, $r$, and $s$ such that $n < ap
< 3n/2 < aq < 2n$ and simultaneously $n < r < 3n/2 <
s < 2n$.
\end{Theorem}

As we have stated at the beginning of this paper (see Sect. \ref{intro}), in 2006 M. El Bachraoui proved that there exists a prime in the interval $[2n,3n]$. Moreover, in 2011 Andy Loo proved that there is a prime in the interval $[3n,4n]$. Now, Theorem \ref{Theorem3} implies these results obtained by M. El Bachraoui and Andy Loo, since that theorem implies that for every integer $2n>14.4$ there exist prime numbers $r$ and $s$ such that $2n<r<3n<s<4n$. In the following table we also show that the interval $[2n,4n]$ contains at least two prime numbers $r$ and $s$ such that $2n<r<3n<s<4n$ for every integer $2n$ such that $2 \le 2n \le 14$:

\begin{longtable}{|c|c|c|c|c|}
\hline
   $2n$ & $r$ & $3n$ & $s$ & $4n$ \\
\hline \hline
\endhead
\multicolumn{5}{c}{Continued on next page...}
\endfoot
\endlastfoot
\hline
   {\bf 2} &          2, 3 &  {\bf 3} & 3 & {\bf 4} \\
\hline
   {\bf 4} &          5 & {\bf 6} & 7 & {\bf 8} \\
\hline
   {\bf 6} &          7 & {\bf 9} & 11 & {\bf 12} \\
\hline
   {\bf 8} &          11 & {\bf 12} & 13 & {\bf 16} \\
\hline
   {\bf 10} &        11, 13 &    {\bf 15} & 17, 19 & {\bf 20} \\
\hline
{\bf 12} &       13, 17 & {\bf 18} & 19, 23 & {\bf 24} \\
\hline
{\bf 14} &       17, 19 & {\bf 21} & 23 & {\bf 28} \\
\hline
\end{longtable}

By using the interval $[n,9(n+3)/8]$ and doing some manual calculations, we can also prove the following:

\begin{itemize}

\item There exists a prime in the interval $[4n,5n]$ for every positive integer $n \neq 2$.

It is easy to prove that this is true if we take into account that $9(4n+3)/8\le5n$ when $n\ge27/4=6.75$.
 
\item There exists a prime in the interval $[5n,6n]$ for every positive integer $n$.

It is easy to verify that this is true if we take into account that $9(5n+3)/8\le6n$ when $n\ge9$.

\item There is a prime in the interval $[6n,7n]$ for every positive integer $n \neq 4$.

This can be easily proved if we consider that $9(6n+3)/8\le7n$ when $n\ge27/2=13.5$.

\item There is a prime in the interval $[7n,8n]$ for every positive integer $n \neq 2$.

This can be easily proved if we consider that $9(7n+3)/8\le8n$ when $n\ge27$.

\end{itemize}

Of course, other intervals can be obtained by using Breusch's interval $[n,9(n+3)/8]$ or Dusart's intervals $[n,(1+1/(2 \ln^2 n))n]$ and $[n,(1+1/(25 \ln^2 n))n]$, for example $[2n,2.25n+3.375]$ and $[4n,4.5n+3.375]$.

\section{Theorems \ref{Theorem7}, \ref{Theorem8}, and \ref{Theorem9}}
\label{sec:10}

Let us suppose that $n$ is a positive integer. We need to know the value that the number $n$ can have so that there is always a perfect square $a^2$ such that $n<a^2<3n/2$.

\begin{Remark}
We say a number is a \emph{perfect square} if it is the square of an integer. In other words, a number $x$ is a perfect square if $\sqrt{x}$ is an integer. Perfect squares are also called \emph{square numbers}.
\end{Remark}

\begin{Remark}
Throughout this paper, we use the symbol $||$ to denote the \emph{absolute value} of a number or expression. For example, the absolute value of $x$ is denoted by $|x|$.
\end{Remark}

In general, if $m$ is any positive integer, we need to know the value that the number $n$ can have so that there is always a positive integer $a$ such that $n<a^m<3n/2$.

We have $n<a^m<3n/2$. This means that
\begin{equation*}
\left|\sqrt[m]{n}\right|<a<\left|\sqrt[m]{\frac{3n}{2}}\right|\text{.}
\end{equation*}

As we said before, the number $a$ is a positive integer. Now, the integer immediately following the number $|\sqrt[m]{n}|$ will be denoted by $|\sqrt[m]{n}|+d$. In other words, $|\sqrt[m]{n}|+d$ is the smallest integer that is greater than $|\sqrt[m]{n}|$:
\begin{itemize}
\item If $|\sqrt[m]{n}|$ is an integer, then $|\sqrt[m]{n}|+d=|\sqrt[m]{n}|+1$, because in this case we have $d=1$.
\item If $|\sqrt[m]{n}|$ is not an integer, then in the expression $|\sqrt[m]{n}|+d$ we have $0<d<1$. We do not have any way of knowing the exact value of $d$ if we do not know the value of $|\sqrt[m]{n}|$ first.
\end{itemize}

Now, let us make the calculation. To start with,
\begin{align*}
\left|\sqrt[{m}]{n}\right| +d &< \left|\sqrt[{m}]{\frac{3n}{2} }\right|\text{.}\\
\intertext{We need to pick the largest possible value of $d$, which is $d=1$ (if ${|\sqrt[{m}]{n}|+1<|\sqrt[{m}]{3n/2}|}$, then $|\sqrt[{m}]{n}|+d<|\sqrt[{m}]{3n/2}|$ for all $d$ such that $0<d\leq1$):}
\left|\sqrt[{m}]{n}\right| +1 &< \left|\sqrt[{m}]{\frac{3n}{2} }\right|\\
1 &< \left|\sqrt[{m}]{\frac{3n}{2} }\right| -\left|\sqrt[{m}]{n}\right|\\
1 &< \left|\sqrt[{m}]{\frac{3}{2} n}\right| -\left|\sqrt[{m}]{n}\right|\\
1 &< \left|\sqrt[{m}]{1.5n}\right| -\left|\sqrt[{m}]{n}\right|\\
1 &< \left|\sqrt[{m}]{1.5}\right| \left|\sqrt[{m}]{n}\right| -\left|\sqrt[{m}]{n}\right|\\
1 &< \left|\sqrt[{m}]{n}\right| \left(\left|\sqrt[{m}]{1.5}\right| -1\right)\\
\left|\sqrt[{m}]{n}\right| &> \frac{1}{\left|\sqrt[{m}]{1.5}\right| -1}\\
n &> \left(\frac{1}{\left|\sqrt[{m}]{1.5}\right| -1} \right)^{m}\\
n &> \frac{1^{m} }{\left(\left|\sqrt[{m}]{1.5}\right| -1\right)^{m} }\\
n &> \frac{1}{\left(\left|\sqrt[{m}]{1.5}\right| -1\right)^{m} }\text{.}
\end{align*}
This means that if $m$ is a positive integer, then for every positive integer $n>1/{\left(\left|\sqrt[{m}]{1.5}\right| -1\right)^{m}}$ there is at least one positive integer $a$ such that ${n<a^{m} <3n/2}$. Now, if $n>14.4/{\left(\left|\sqrt[{m}]{1.5}\right| -1\right)^{m} } $, then $n>1/{\left(\left|\sqrt[{m}]{1.5}\right| -1\right)^{m} }$. As a consequence, we state the following theorem:

\begin{Theorem}
\label{Theorem7}
If $m$ is a positive integer, then for every positive integer $n>14.4/(|\sqrt[m]{1.5}|-1)^m$ there is at least one positive integer $a$ such that $n<a^m<3n/2$.
\end{Theorem}

Now we are going to prove the following proposition:

\begin{Proposition}
\label{proposition3}
If $m$ is any positive integer, then $1/(|\sqrt[m]{1.5}|-1)^m>1$.
\end{Proposition}

\begin{proof}
\begin{align*}
\frac{1}{\left(\left|\sqrt[{m}]{1.5}\right| -1\right)^{m} }  &> 1\\
1 &> 1\left(\left|\sqrt[{m}]{1.5}\right| -1\right)^{m}\\
1 &> \left(\left|\sqrt[{m}]{1.5}\right| -1\right)^{m}\\
\left|\sqrt[{m}]{1}\right|  &> \left|\sqrt[{m}]{1.5}\right| -1\\
1 &> \left|\sqrt[{m}]{1.5}\right| -1\\
1+1 &> \left|\sqrt[{m}]{1.5}\right|\\
2 &> \left|\sqrt[{m}]{1.5}\right|\\
2^{m}  &> 1.5
\end{align*}
It is very easy to verify that $2^m>1.5$ for every positive integer $m$. Consequently, if $m$ is any positive integer, then $1/{\left(\left|\sqrt[{m}]{1.5}\right| -1\right)^{m} } >1$.
\end{proof}

\begin{Remark}
In general, to prove that an inequality is correct, we can solve that inequality step by step. If we obtain a result that is obviously correct, then we can start with that correct result, `work backwards from there' and prove that the initial statement is true.
\end{Remark}

If ${1/{\left(\left|\sqrt[{m}]{1.5}\right| -1\right)^{m} } >1}$, then $14.4/{\left(\left|\sqrt[{m}]{1.5}\right| -1\right)^{m} } >14.4$. This means that $14.4/{\left(\left|\sqrt[{m}]{1.5}\right| -1\right)^{m} } >14.4$ for every positive integer $m$. As a result, we state the following theorem:

\begin{Theorem}
\label{Theorem8}
If $m$ is any positive integer and \mbox{$n>14.4/{{{{\left( {\left|\sqrt[m]{{1.5}}\right| - 1} \right)}^m}}}$,} then $n > 14.4$.
\end{Theorem}

In Sect. \ref{sec:7} we proved that for every positive integer $n>14.4$ there exist prime numbers $r$ and $s$ such that $n<r<3n/2<s<2n$. This true statement is called \emph{Theorem \ref{Theorem3}}. Now, the following theorem is deduced from Theorems \ref{Theorem3}, \ref{Theorem7}, and \ref{Theorem8}:

\begin{Theorem}
\label{Theorem9}
If $m$ is a positive integer, then for
every positive integer $n > 14.4/{{{{\left( {\left|\sqrt[m]{{1.5}}\right| - 1}
\right)}^m}}}$ there exist a positive integer $a$ and a prime number
$s$ such that $n < {a^m} < 3n/2 < s < 2n$.
\end{Theorem}

\section{Theorems \ref{Theorem10}, \ref{Theorem11}, and \ref{Theorem12}}
\label{sec:11}

Taking into account that $m$ is a positive integer, let us find out the value that the number $n$ can have so that there is always a positive integer $a$ such that
$3n/2 < {a^m} < 2n$.

We have
\begin{align*}
\frac{{3n}}{2} &< {a^m} < 2n\text{,}\\
\intertext{which means that}
\left|\sqrt[m]{{\frac{{3n}}{2}}}\right| &< a < \left|\sqrt[m]{{2n}}\right|\text{.}
\end{align*}
The integer immediately
following the number $|\sqrt[m]{{3n/2}}|$ will be denoted by $|\sqrt[m]{{3n/2}}| + d$. In other words, $|\sqrt[m]{{3n/2}}| +
d$ is the smallest integer that is greater than $|\sqrt[m]{{3n/2}}|$:
\begin{itemize}
	\item If $|\sqrt[m]{{3n/2}}|$ is an integer, then $|\sqrt[m]{{3n/2}}|
+ d = |\sqrt[m]{{3n/2}}| + 1$, because in this case we have $d = 1$.
	\item If $|\sqrt[m]{{3n/2}}|$ is not an integer, then in the expression
$|\sqrt[m]{{3n/2}}| + d$ we have $0 < d < 1$. We do not have any way
of knowing the exact value of $d$ if we do not know the value of
$|\sqrt[m]{{3n/2}}|$ first.
\end{itemize}

Let us make the calculation. To begin with,
\begin{align*}
\left|\sqrt[m]{{\frac{{3n}}{2}}}\right| + d  &<  \left|\sqrt[m]{{2n}}\right|\text{.}\\
\intertext{We need to pick the largest possible value of $d$, which is $d=1$
(if ${|\sqrt[m]{{3n/2}}| + 1 < |\sqrt[m]{{2n}}|}$, then
$|\sqrt[m]{{3n/2}}| + d < |\sqrt[m]{{2n}}|$ for all $d$ such that $0
< d \le 1$):}
\left|\sqrt[m]{{\frac{{3n}}{2}}}\right| + 1  &<  \left|\sqrt[m]{{2n}}\right|\\
1  &<  \left|\sqrt[m]{{2n}}\right| - \left|\sqrt[m]{{\frac{{3n}}{2}}}\right|\\
1  &<  \left|\sqrt[m]{{2n}}\right| - \left|\sqrt[m]{{\frac{3}{2}n}}\right|\\
1  &<  \left|\sqrt[m]{{2n}}\right| - \left|\sqrt[m]{{1.5n}}\right|\\
1  &<  \left|\sqrt[m]{2}\right|\left|\sqrt[m]{n}\right| - \left|\sqrt[m]{{1.5}}\right|\left|\sqrt[m]{n}\right|\\
1  &<  \left|\sqrt[m]{n}\right|\left( {\left|\sqrt[m]{2}\right| - \left|\sqrt[m]{{1.5}}\right|} \right)\\
\left|\sqrt[m]{n}\right|  &>  \frac{1}{{\left|\sqrt[m]{2}\right| - \left|\sqrt[m]{{1.5}}\right|}}\\
n  &>  {\left( {\frac{1}{{\left|\sqrt[m]{2}\right| - \left|\sqrt[m]{{1.5}}\right|}}} \right)^m}\\
n  &>  \frac{{{1^m}}}{{{{\left( {\left|\sqrt[m]{2}\right| - \left|\sqrt[m]{{1.5}}\right|} \right)}^m}}}\\
n  &>  \frac{{{1^{}}}}{{{{\left( {\left|\sqrt[m]{2}\right| - \left|\sqrt[m]{{1.5}}\right|} \right)}^m}}}\text{.}
\end{align*}
This means that if $m$ is a positive integer, then for every positive
integer $n > 1/{{{{\left( {\left|\sqrt[m]{2}\right| - \left|\sqrt[m]{{1.5}}\right|} \right)}^m}}}$
there is at least one positive integer $a$ such that \mbox{$3n/2 <
{a^m} < 2n$}. Now, if $n > 14.4/{{{{\left( {\left|\sqrt[m]{2}\right| - \left|\sqrt[m]{{1.5}}\right|}
\right)}^m}}}$, then $n > 1/{{{{\left( {\left|\sqrt[m]{2}\right| - \left|\sqrt[m]{{1.5}}\right|}
\right)}^m}}}$. As a consequence, we state the following theorem:

\begin{Theorem}
\label{Theorem10}
If $m$ is a positive integer, then for every positive
integer \mbox{$n > 14.4/{{{{\left( {\left|\sqrt[m]{2}\right| - \left|\sqrt[m]{{1.5}}\right|}
\right)}^m}}}$} there is at least one positive integer $a$ such that \mbox{$3n/2 < {a^m} < 2n$.}
\end{Theorem}

Now we are going to prove the following proposition:

\begin{Proposition}
\label{Proposition4}
If $m$ is any positive integer, then
$1/{{{{\left( {\left|\sqrt[m]{2}\right| - \left|\sqrt[m]{{1.5}}\right|} \right)}^m}}} > 1$.
\end{Proposition}

\begin{proof}

\begin{align*}
\frac{1}{{{{\left( {\left|\sqrt[m]{2}\right| - \left|\sqrt[m]{{1.5}}\right|} \right)}^m}}}  &>  1\\
1  &>  1{\left( {\left|\sqrt[m]{2}\right| - \left|\sqrt[m]{{1.5}}\right|} \right)^m}\\
1  &>  {\left( {\left|\sqrt[m]{2}\right| - \left|\sqrt[m]{{1.5}}\right|} \right)^m}\\
\left|\sqrt[m]{1}\right|  &>  \left|\sqrt[m]{2}\right| - \left|\sqrt[m]{{1.5}}\right|\\
1  &>  \left|\sqrt[m]{2}\right| - \left|\sqrt[m]{{1.5}}\right|\\
1 + \left|\sqrt[m]{{1.5}}\right|  &>  \left|\sqrt[m]{2}\right|\\
\intertext{Now, we know that $|\sqrt[m]{{1.5}}| > 1$, since $1.5 > {1^m}$, that is to say, $1.5 > 1$. On the other hand, we have $|\sqrt[m]{2}| \le 2$, since $2 \le {2^m}$. This means that}
1 + \left|\sqrt[m]{{1.5}}\right|  &>  2  \ge  \left|\sqrt[m]{2}\right|\text{,}\\
\intertext{which proves that}
1 + \left|\sqrt[m]{{1.5}}\right|  &>  \left|\sqrt[m]{2}\right|\text{.}
\end{align*}
Therefore, if $m$ is any positive integer, then $1/{{{{\left(
{\left|\sqrt[m]{2}\right| - \left|\sqrt[m]{{1.5}}\right|} \right)}^m}}} > 1$.
\end{proof}

If $1/{{{{\left(
{\left|\sqrt[m]{2}\right| - \left|\sqrt[m]{{1.5}}\right|} \right)}^m}}} > 1$, then $14.4/{{{{\left(
{\left|\sqrt[m]{2}\right| - \left|\sqrt[m]{{1.5}}\right|} \right)}^m}}} > 14.4$. This means that
$14.4/{{{{\left( {\left|\sqrt[m]{2}\right| - \left|\sqrt[m]{{1.5}}\right|} \right)}^m}}} > 14.4$
for every positive integer $m$. As a consequence, we state the following theorem:

\begin{Theorem}
\label{Theorem11}
If $m$ is any positive integer and $n$ is a number such that $n>14.4/\left(\left|\sqrt[m]{2}\right|-\left|\sqrt[m]{1.5}\right|\right)^m$, then $n>14.4$.
\end{Theorem}

The following theorem is deduced from Theorems \ref{Theorem3}, \ref{Theorem10}, and \ref{Theorem11}:

\begin{Theorem}
\label{Theorem12}
If $m$ is a positive integer, then for
every positive integer $n > 14.4/{{{{\left( {\left|\sqrt[m]{2}\right| -
\left|\sqrt[m]{{1.5}}\right|} \right)}^m}}}$ there exist a prime number $r$ and a
positive integer $a$ such that $n < r < 3n/2 < {a^m} < 2n$.
\end{Theorem}

\section{Conclusion}
\label{sec:13}

We will restate the three most important theorems that have been proved in this paper (see Sects. \ref{sec:7}, \ref{sec:10}, and \ref{sec:11}):

\begin{center}
\fbox{
\parbox[c][2.5cm][c]{10cm}{\centering Theorem \ref{Theorem4}: If $a$ is a positive integer greater than
1, then for every positive integer $n > 14.4a$ there are always at least four
prime numbers $p$, $q$, $r$, and $s$ such that $n < ap
< 3n/2 < aq < 2n$ and simultaneously $n < r < 3n/2 <
s < 2n$.}
}
\end{center}
\begin{center}
\fbox{
\parbox[c][2cm][c]{10cm}{\centering Theorem \ref{Theorem9}: If $m$ is a positive integer, then for
every positive integer $n > 14.4/{{{{\left( {\left|\sqrt[m]{{1.5}}\right| - 1}
\right)}^m}}}$ there exist a positive integer $a$ and a prime number
$s$ such that $n < {a^m} < 3n/2 < s < 2n$.}
}
\end{center}
\begin{center}
\fbox{
\parbox[c][2cm][c]{10cm}{\centering Theorem \ref{Theorem12}: If $m$ is a positive integer, then for
every positive integer $n > 14.4/{{{{\left( {\left|\sqrt[m]{2}\right| -
\left|\sqrt[m]{{1.5}}\right|} \right)}^m}}}$ there exist a prime number $r$ and a
positive integer $a$ such that $n < r < 3n/2 < {a^m} < 2n$.}
}
\end{center}


\begin{thebibliography}{99}

\bibitem{Dusart}
P. Dusart, Autour de la fonction qui compte le nombre de nombres premiers, \emph{Ph.D. Thesis}, Universit\'e de Limoges, (1998).

\bibitem{Dusart2}
P. Dusart, \emph{Estimates of Some Functions Over Primes without R.H.}, arXiv:1002.0442 [math.NT], (2010).

\bibitem{M. El Bachraoui}
M. El Bachraoui, Primes in the interval $[2n,3n]$, \emph{Int. J. Contemp. Math. Sciences}, 1(13) (2006), 617-621.

\bibitem{Loo}
A. Loo, On the primes in the interval $[3n,4n]$, \emph{Int. J. Contemp. Math. Sciences}, 6(38) (2011), 1871-1882.

\end{thebibliography}
\end{document}